\documentclass{amsart}

\usepackage{amssymb} \usepackage{amsfonts} \usepackage{amsmath}
\usepackage{amsthm} \usepackage{epsfig} \usepackage{wrapfig}

\newtheorem*{theorem}{Theorem}
\newtheorem*{prop}{Proposition}
\newtheorem*{corollary}{Corollary}

\author{Evgeny Fominykh}
\author{Andrei Vesnin}

\address{Chelyabinsk State University, Chelyabinsk \newline
 \hspace*{10pt} Institute of Mathematics and Mechanics, Ekaterinburg}

 \email{fominykh@csu.ru}

\address{Sobolev Institute of Mathematics, Novosibirsk}

 \email{vesnin@math.nsc.ru}

\title{Exact values of complexity for Paoluzzi~-- Zimmermann manifolds}



\thanks{The authors were supported by the Russian Foundation for Basic Research,
grants 10-01-91056 (A. Vesnin) and 11-01-00605 (E. Fominykh), and
the Joint Program of the Institute of Mathematics and Mechanics of
the Ural Branch of the Russian Academy of Sciences and the Institute
of Mathematics of the Siberian Branch of the Russian Academy of
Sciences.}

\begin{document}

\begin{abstract}
 There are found exact values of (Matveev) complexity for the 2-parameter
family of hyperbolic 3-manifolds with boundary constructed by
Paoluzzi and Zimmermann. Moreover, $\varepsilon$-invariants for
these manifolds are calculated.
\end{abstract}

\maketitle

\section{Preliminaries}

 Let $M$ be a compact $3$-manifold with nonempty boundary. Recall
\cite{MatBook} that a subpolyhedron $P\subset M$ is a \emph{spine}
of $M$ if the manifold $M\setminus P$ is homeomorphic to $\partial M
\times (0,1]$. A spine $P$ is said to be \emph{almost simple} if the
link of every of its points can be embedded into a complete graph
$K_4$ with four vertices. A \emph{true vertex} of an almost simple
spine $P$ is a point with the link $K_4$. The \emph{complexity}
$c(M)$ of $M$ is defined as the minimum possible number of true
vertices of an almost simple spine of $M$.

 An almost simple spine $P$ is called \emph{simple} if the link of
each point $x\in P$ is homeomorphic to one of the following
$1$-dimensional polyhedra: (a) a circle (such a point $x$ is called
nonsingular); (b) a circle with a diameter (such an $x$ is a triple
point); (c) $K_4$. Components of the set of nonsingular points are
said to be $2$-\emph{components} of $P$, while components of the set
of triple points are said to be \emph{triple lines} of $P$. A simple
spine is \emph{special} if each of its triple lines is an open
$1$-cell and each of its $2$-components is an open $2$-cell.

 The problem of calculating the complexity for $3$-manifolds is
actual but it is very hard. Exact values of the complexity are
presently known only for some computer-generated
censuses~\cite{Matveev05}, for two infinite series of hyperbolic
manifolds with boundary~\cite{Anisov05, FrMartPetr03}, for certain
infinite series of lens spaces, and for generalized quaternion
spaces~\cite{JacoRubTill09, JacoRubTillPrepr}.

 In the paper we find the exact values of the complexity for a family
of manifolds $M_{n,k}$ with boundary constructed by Paoluzzi and
Zimmermann~\cite{PaolZim96}. Every $M_{n,k}$, $n \geqslant 4$, has
one boundary component, its complexity $c(M_{n,k})$ is equal to $n$,
and its Euler characteristic is equal to $2-n$. Earlier known
examples of manifolds with one boundary component and complexity $n$
\cite{FrMartPetr03} have the Euler characteristic $1-n$.

 To prove the main result we use the $\varepsilon$-invariant of
$3$-manifolds (see~\cite[section 8.1.3]{MatBook}). Let $P$ be a
special spine of a compact manifold $M$. Denote by $\mathcal{F}(P)$
the set of all simple subpolyhedra of $P$ including $P$ and the
empty set. Let us associate to each simple polyhedron $Q\subset P$
its $\varepsilon$-weight $w_{\varepsilon}(Q) =
(-1)^{V(Q)}\varepsilon^{\chi(Q)-V(Q)}$, where $V(Q)$ is the number
of vertices of $Q$, $\chi(Q)$ is its Euler characteristic, and
$\varepsilon = (1+\sqrt{5})/2$ is a solution of the equation
$\varepsilon^2=\varepsilon+1$. The $\varepsilon$-invariant $t(M)$ of
$M$ is given by the formula $t(M) = \sum_{Q\in \mathcal{F} (P)}
w_{\varepsilon}(Q)$.

\section{Paoluzzi-Zimmermann manifolds}

For every $n\geqslant 3$ consider an $n$-gonal bipyramid which is
the union of pyramids $N L_{0} L_{1} \ldots L_{n-1}$ and $S L_{0}
L_{1} \ldots L_{n-1}$ along the common $n$-gonal base $L_{0} L_{1}
\ldots L_{n-1}$. Let $k$ be such integer that $0 \leqslant k < n$
and $\gcd (n, 2-k)=1$. We identify the faces of $\mathcal B_{n}$ in
pairs: for each $i = 0, \ldots, n-1$ the face $L_{i} L_{i+1} N$ gets
identified with the face $S L_{i+k} L_{i+k+1}$ by a transformation
(homeomorphism of faces) which we shall denote by $y_{i}$ (indices
are taken $\mod n$ and the vertices are glued together in the order
in which they are written).

\begin{figure}[htb]
\centering \unitlength=0.4mm
\begin{picture}(300,205)(0,-40)
\thicklines \qbezier(0,10)(0,10)(60,10)
\qbezier(60,10)(60,10)(120,10) \qbezier(120,10)(120,10)(180,10)
\qbezier(180,10)(180,10)(240,10) \qbezier(240,10)(240,10)(300,10)
\qbezier(0,10)(0,10)(0,40) \qbezier(60,10)(60,10)(60,40)
\qbezier(120,10)(120,10)(120,40) \qbezier(180,10)(180,10)(180,40)
\qbezier(240,10)(240,10)(240,40) \qbezier(300,10)(300,10)(300,40)
\qbezier(0,110)(0,110)(0,80) \qbezier(60,110)(60,110)(60,80)
\qbezier(120,110)(120,110)(120,80)
\qbezier(180,110)(180,110)(180,80)
\qbezier(240,110)(240,110)(240,80)
\qbezier(300,110)(300,110)(300,80)
\qbezier(0,40)(0,40)(-20,60) \qbezier(0,40)(0,40)(20,60)
\qbezier(0,80)(0,80)(-20,60) \qbezier(0,80)(0,80)(20,60)
\qbezier(20,60)(20,60)(40,60) \qbezier(60,40)(60,40)(40,60)
\qbezier(60,40)(60,40)(80,60) \qbezier(60,80)(60,80)(40,60)
\qbezier(60,80)(60,80)(80,60) \qbezier(80,60)(80,60)(100,60)
\qbezier(120,40)(120,40)(100,60) \qbezier(120,40)(120,40)(140,60)
\qbezier(120,80)(120,80)(100,60) \qbezier(120,80)(120,80)(140,60)
\qbezier(140,60)(140,60)(160,60) \qbezier(180,40)(180,40)(160,60)
\qbezier(180,40)(180,40)(200,60) \qbezier(180,80)(180,80)(160,60)
\qbezier(180,80)(180,80)(200,60) \qbezier(200,60)(200,60)(220,60)
\qbezier(240,40)(240,40)(220,60) \qbezier(240,40)(240,40)(260,60)
\qbezier(240,80)(240,80)(220,60) \qbezier(240,80)(240,80)(260,60)
\qbezier(260,60)(260,60)(280,60) \qbezier(300,40)(300,40)(280,60)
\qbezier(300,80)(300,80)(280,60)
\qbezier(0,110)(0,110)(180,110) \qbezier(180,110)(180,110)(300,110)
{\thinlines \qbezier(0,110)(0,110)(15,115)
\qbezier(30,120)(30,120)(45,125) \qbezier(60,130)(60,130)(75,135)
\qbezier(90,140)(90,140)(105,145)
\qbezier(120,150)(120,150)(150,160)
\qbezier(60,110)(60,110)(69,115) \qbezier(78,120)(78,120)(87,125)
\qbezier(96,130)(96,130)(105,135)
\qbezier(114,140)(114,140)(123,145)
\qbezier(132,150)(132,150)(150,160)
\qbezier(120,110)(120,110)(123,115)
\qbezier(126,120)(126,120)(129,125)
\qbezier(132,130)(132,130)(135,135)
\qbezier(138,140)(138,140)(141,145)
\qbezier(144,150)(144,150)(150,160)
\qbezier(180,110)(180,110)(177,115)
\qbezier(174,120)(174,120)(171,125)
\qbezier(168,130)(168,130)(165,135)
\qbezier(162,140)(162,140)(159,145)
\qbezier(156,150)(156,150)(150,160)
\qbezier(240,110)(240,110)(231,115)
\qbezier(222,120)(222,120)(213,125)
\qbezier(204,130)(204,130)(195,135)
\qbezier(186,140)(186,140)(177,145)
\qbezier(168,150)(168,150)(150,160)
\qbezier(300,110)(300,110)(285,115)
\qbezier(270,120)(270,120)(255,125)
\qbezier(240,130)(240,130)(225,135)
\qbezier(210,140)(210,140)(195,145)
\qbezier(180,150)(180,150)(150,160)
\qbezier(0,10)(0,10)(15,5) \qbezier(30,0)(30,0)(45,-5)
\qbezier(60,-10)(60,-10)(75,-15) \qbezier(90,-20)(90,-20)(105,-25)
\qbezier(120,-30)(120,-30)(150,-40)
\qbezier(60,10)(60,10)(69,5) \qbezier(78,0)(78,0)(87,-5)
\qbezier(96,-10)(96,-10)(105,-15)
\qbezier(114,-20)(114,-20)(123,-25)
\qbezier(132,-30)(132,-30)(150,-40)
\qbezier(120,10)(120,10)(123,5) \qbezier(126,0)(126,0)(129,-5)
\qbezier(132,-10)(132,-10)(135,-15)
\qbezier(138,-20)(138,-20)(141,-25)
\qbezier(144,-30)(144,-30)(150,-40)
\qbezier(180,10)(180,10)(177,5) \qbezier(174,0)(174,0)(171,-5)
\qbezier(168,-10)(168,-10)(165,-15)
\qbezier(162,-20)(162,-20)(159,-25)
\qbezier(156,-30)(156,-30)(150,-40)
\qbezier(240,10)(240,10)(231,5) \qbezier(222,0)(222,0)(213,-5)
\qbezier(204,-10)(204,-10)(195,-15)
\qbezier(186,-20)(186,-20)(177,-25)
\qbezier(168,-30)(168,-30)(150,-40)
\qbezier(300,10)(300,10)(285,5) \qbezier(270,0)(270,0)(255,-5)
\qbezier(240,-10)(240,-10)(225,-15)
\qbezier(210,-20)(210,-20)(195,-25)
\qbezier(180,-30)(180,-30)(150,-40) }
\put(150,165){\makebox(0,0)[cc]{$N$}}
\put(150,-45){\makebox(0,0)[cc]{$S$}} {\thinlines
\qbezier(0,40)(0,40)(0,50) \qbezier(0,55)(0,55)(0,65)
\qbezier(0,70)(0,70)(0,80) \qbezier(-20,60)(-20,60)(-10,60)
\qbezier(-5,60)(-5,60)(5,60) \qbezier(10,60)(10,60)(20,60)
\put(-5,65){\makebox(0,0)[cc]{$L_{4}$}}
\qbezier(60,40)(60,40)(60,50) \qbezier(60,55)(60,55)(60,65)
\qbezier(60,70)(60,70)(60,80) \qbezier(40,60)(40,60)(50,60)
\qbezier(55,60)(55,60)(65,60) \qbezier(70,60)(70,60)(80,60)
\put(55,65){\makebox(0,0)[cc]{$L_{0}$}}
\qbezier(120,40)(120,40)(120,50) \qbezier(120,55)(120,55)(120,65)
\qbezier(120,70)(120,70)(120,80) \qbezier(100,60)(100,60)(110,60)
\qbezier(115,60)(115,60)(125,60) \qbezier(130,60)(130,60)(140,60)
\put(115,65){\makebox(0,0)[cc]{$L_{1}$}}
\qbezier(180,40)(180,40)(180,50) \qbezier(180,55)(180,55)(180,65)
\qbezier(180,70)(180,70)(180,80) \qbezier(160,60)(160,60)(170,60)
\qbezier(175,60)(175,60)(185,60) \qbezier(190,60)(190,60)(200,60)
\put(175,65){\makebox(0,0)[cc]{$L_{2}$}}
\qbezier(240,40)(240,40)(240,50) \qbezier(240,55)(240,55)(240,65)
\qbezier(240,70)(240,70)(240,80) \qbezier(220,60)(220,60)(230,60)
\qbezier(235,60)(235,60)(245,60) \qbezier(250,60)(250,60)(260,60)
\put(235,65){\makebox(0,0)[cc]{$L_{3}$}} }
\put(0,120){\makebox(0,0)[cc]{$A_4$}}
\put(60,120){\makebox(0,0)[cc]{$A_0$}}
\put(120,120){\makebox(0,0)[cc]{$A_1$}}
\put(180,120){\makebox(0,0)[cc]{$A_2$}}
\put(240,120){\makebox(0,0)[cc]{$A_3$}}
\put(300,120){\makebox(0,0)[cc]{$A_4$}}
\put(0,0){\makebox(0,0)[cc]{$F_3$}}
\put(60,0){\makebox(0,0)[cc]{$F_4$}}
\put(120,0){\makebox(0,0)[cc]{$F_0$}}
\put(180,0){\makebox(0,0)[cc]{$F_1$}}
\put(240,0){\makebox(0,0)[cc]{$F_2$}}
\put(300,0){\makebox(0,0)[cc]{$F_3$}}
\put(10,80){\makebox(0,0)[cc]{$B_4$}}
\put(70,80){\makebox(0,0)[cc]{$B_0$}}
\put(130,80){\makebox(0,0)[cc]{$B_1$}}
\put(190,80){\makebox(0,0)[cc]{$B_2$}}
\put(250,80){\makebox(0,0)[cc]{$B_3$}}
\put(310,80){\makebox(0,0)[cc]{$B_4$}}
\put(10,40){\makebox(0,0)[cc]{$E_3$}}
\put(70,40){\makebox(0,0)[cc]{$E_4$}}
\put(130,40){\makebox(0,0)[cc]{$E_0$}}
\put(190,40){\makebox(0,0)[cc]{$E_1$}}
\put(250,40){\makebox(0,0)[cc]{$E_2$}}
\put(310,40){\makebox(0,0)[cc]{$E_3$}}
\put(20,50){\makebox(0,0)[cc]{$C_4$}}
\put(80,50){\makebox(0,0)[cc]{$C_0$}}
\put(140,50){\makebox(0,0)[cc]{$C_1$}}
\put(200,50){\makebox(0,0)[cc]{$C_2$}}
\put(260,50){\makebox(0,0)[cc]{$C_3$}}
\put(-20,50){\makebox(0,0)[cc]{$D_3$}}
\put(40,50){\makebox(0,0)[cc]{$D_4$}}
\put(100,50){\makebox(0,0)[cc]{$D_0$}}
\put(160,50){\makebox(0,0)[cc]{$D_1$}}
\put(220,50){\makebox(0,0)[cc]{$D_2$}}
\put(280,50){\makebox(0,0)[cc]{$D_3$}}
\put(30,90){\makebox(0,0)[cc]{$\mathcal X_4$}}
\put(90,90){\makebox(0,0)[cc]{$\mathcal X_0$}}
\put(150,90){\makebox(0,0)[cc]{$\mathcal X_1$}}
\put(210,90){\makebox(0,0)[cc]{$\mathcal X_2$}}
\put(270,90){\makebox(0,0)[cc]{$\mathcal X_3$}}
\put(30,30){\makebox(0,0)[cc]{$\mathcal X_4'$}}
\put(90,30){\makebox(0,0)[cc]{$\mathcal X_0'$}}
\put(150,30){\makebox(0,0)[cc]{$\mathcal X_1'$}}
\put(210,30){\makebox(0,0)[cc]{$\mathcal X_2'$}}
\put(270,30){\makebox(0,0)[cc]{$\mathcal X_3'$}}
\end{picture}
\caption{ The truncated bipyramid $\mathcal A_{5}$.} \label{pyramid}
\end{figure}
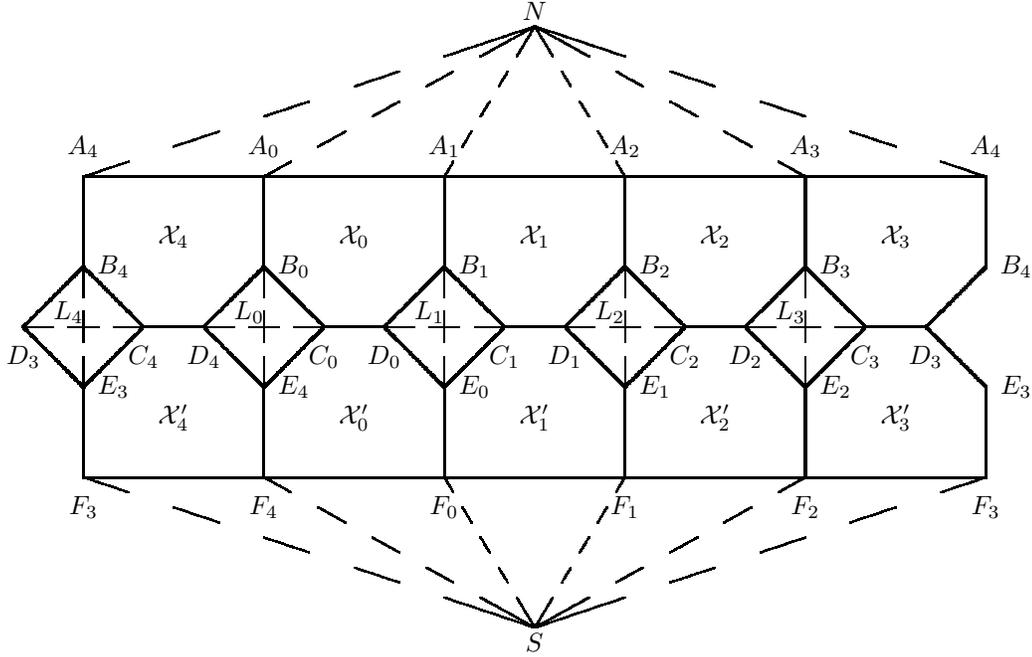

 Identifications $\{ y_{0}, y_{1}, \ldots , y_{n-1} \}$ define the
equivalence relations on the sets of faces, edges, and vertices of
the bipyramid. It is easy to see that all the faces are split into
pairs equivalent faces, all edges and all vertices become identified
to a single edge resp. vertex (this is guaranteed by the above
conditions on $k$). Denote the resulting identification spaces by
$M_{n,k}^{*}$. It is an orientable pseudomanifold with one singular
point, since $\chi(M_{n,k}^{*}) = 1 - 1 + n - 1 = n-1 \neq 0$.
Cutting of a cone neighborhood of the singular point from
$M_{n,k}^{*}$ we get a compact manifold $M_{n,k}$ with one boundary
component. It arises from the truncated bipyramid $\mathcal A_{n}$
with vertices $A_{i}$, $B_{i}$, $C_{i}$, $D_{i}$, $E_{i}$, $F_{i}$,
$i = 0, \ldots, n-1$ (see Fig.~\ref{pyramid} for the case $n=5$,
where the right and left sides are assumed identified). Pairwise
identifications $\{ y_{0}, y_{1}, \ldots , y_{n-1} \}$ of the faces
of $\mathcal B_{n}$ induce pairwise identifications $\{ x_{0},
x_{1}, \ldots x_{n-1}\}$ of the faces of $\mathcal A_{n}$, where
$x_{i}$ identifies faces $\mathcal X_{i} =C_{i} D_{i} B_{i+1}
A_{i+1} A_{i} B_{i} $ and $\mathcal X_{k+i}' = F_{k+i-1} E_{k+i-1}
C_{k+i} D_{k+i} E_{k+i} F_{k+i}$ (indices are taken $\mod n$ and the
vertices are glued together in the order in which they are written).
In this case the faces $A_{0} A_{1} \ldots A_{n-1}$, $F_{0} F_{1}
\ldots F_{n-1}$, and $B_{i} C_{i} E_{i} D_{i-1}$, $i=0, \ldots,
n-1$, form the boundary $\partial M_{n,k}$ of $M_{n,k}$. It is easy
to check that $\partial M_{n,k}$ is a closed orientable surface of
genus $n-1$. According to \cite{PaolZim96}, for every $n \geqslant
3$ the manifold $M_{n,k}$ is hyperbolic and $M_{n,k} = \mathbb H^{3}
/ G_{n,k}$, where $G_{n,k}$ is an $n$-generated group with one
defining relation

$$
G_{n,k} = \langle x_{0}, \ldots, x_{n-1} \, \mid \,
\prod_{i=0}^{n-1} \, x_{i(2-k)} x_{i(2-k)+1}^{-1} x_{(i+1)(2-k)
-1}^{-1} = 1 \rangle.
$$

 It is known \cite{PaolZim96} that $M_{n,k}$ and $M_{n', k'}$ are
homeomorphic (or equivalently, isometric) if and only if $n = n'$
and $k\equiv k' \mod n$. The manifold $M_{3,1}$ is constructed in
\cite{Thurston} by identifying faces of $2$ truncated hyperbolic
tetrahedra. It was shown in \cite{Ko-Mi} that $M_{3,1}$ is one of
the compact manifolds which have minimal volume among all compact
hyperbolic $3$-manifolds with totally geodesic boundary,
$\textrm{vol} (M_{3,1}) \approx 6,451998$. The volumes of $M_{n,k}$
for $n \leqslant 82$ are presented in \cite{Ushijima}.

\section{Calculating the complexity}

\begin{theorem} \label{theorem:main}
For every integer $n \geqslant 4$ we have $c(M_{n,k})=n$.
\end{theorem}

\begin{proof}
 Let us prove that $c(M_{n,k})\leqslant n$ for every integer $n \geqslant 3$.
To do that it suffices to construct a special spine of $M$ with $n$
true vertices. Cut $\mathcal B_{n}$ into $n$ tetrahedra $\mathcal
T_{i} = N S L_{i} L_{i+1}$, where $i = 0, 1, \ldots, n-1$. For each
$\mathcal T_{i}$ consider the union of the links of all four
vertices of $\mathcal T_{i}$ in the first barycentric subdivision.
$M_{n,k}^{*}$ can be obtained by gluing the tetrahedra $\mathcal
T_{0}, \ldots , \mathcal T_{n-1}$ via affine homeomorphisms of the
faces. This gluing determines a pseudotriangulation $\mathcal T$ of
$M_{n,k}^{*}$ and induces a gluing of the corresponding polyhedra
$R_i$, $i = 0, \ldots, n-1$, together. We get a special spine
$P_{n,k} = \cup_i R_i$ of $M_{n,k}$. Since every $R_i$ is
homeomorphic to a cone over $K_4$, the spine $P_{n,k}$ has exactly
$n$ true vertices.

 Demonstrate that the inequality $c(M_{n,k}) \geqslant n$ holds
for every $n \geqslant 4$. Since $M_{n,k}$ is hyperbolic, it is
irreducible, has incompressible boundary, and contains no essential
annuli. It follows from \cite[Theorem 2.2.4]{MatBook} that there
exists a special spine $P'$ of $M_{n,k}$ with $c(M_{n,k})$ true
vertices. Denote by $d$ the number of $2$-components of $P'$.
Calculating the Euler characteristic of $M_{n,k}$, we get
$2-n=\chi(P_{n,k})=\chi(M_{n,k})=\chi(P')=d-c(M_{n,k})$. Hence
$c(M_{n,k})=n+d-2$. It remains to show that $d\geqslant 2$.

 On the contrary, suppose that $d=1$. Let us calculate
the $\varepsilon$-invariant $t(M_{n,k})$ in two ways from the spines
$P_{n,k}$ and $P'$. By construction, $P_{n,k}$ has two
$2$-components, and each of them corresponds to an edge of $\mathcal
T$. Denote by $\alpha$ and $\beta$ $2$-components of $P_{n,k}$,
where $\alpha$ corresponds to the edge $NS$. Note that for
describing an arbitrary simple subpolyhedron $Q$ of $P_{n,k}$ it is
sufficient to specify which $2$-components of $P_{n,k}$ are
contained in $Q$ (see~\cite[section~8.1.3]{MatBook}). A simple
analysis shows that there are three simple subpolyhedra of
$P_{n,k}$.
\begin{enumerate}
 \item The empty subpolyhedron with the $\varepsilon$-weight 1.
 \item The whole $P_{n,k}$ with the $\varepsilon$-weight equals to
$(-1)^n\varepsilon^{2-2n}$, since $V(P_{n,k})=n$ and
$\chi(P_{n,k})=2-n$.
 \item The subpolyhedron $P_{n,k}\setminus \alpha$. It contains
no vertices and has the Euler characteristic $\chi (P_{n,k}\setminus
\alpha) = \chi(P_{n,k})-1=1-n$. Hence it has $\varepsilon$-weight
$\varepsilon^{1-n}$.
\end{enumerate}
Summing up, we get
$t(M_{n,k})=(-1)^n\varepsilon^{2-2n}+\varepsilon^{1-n}+1$.

 Now let us calculate the $\varepsilon$-invariant $t(M_{n,k})$
from $P'$. It is easy to see that $\mathcal{F}(P') = \{\emptyset,
P'\}$ and $\chi(P')=2-n$, since the spine $P'$ has only one
$2$-component and $n-1$ true vertices. Summing up
$w_{\varepsilon}(\emptyset)=1$ and
$w_{\varepsilon}(P')=(-1)^{n-1}\varepsilon^{3-2n}$, we get
$t(M_{n,k}) = (-1)^{n-1} \varepsilon^{3-2n}+1$.

 Comparing two obtained formulas for $t(M_{n,k})$, we get
\linebreak $\varepsilon^{n-1} = (-1)^{n-1}(\varepsilon+1)$ that is
equivalent to $\varepsilon^{n-3} = (-1)^{n-1}$. It is obvious that
this equality is true only for the case $n=3$, which contradicts our
assumption that $n \geqslant 4$. Hence, $d>1$, and $c(M_{n,k})
\geqslant n$.
\end{proof}

\begin{prop}
 The following equality holds: $c(M_{3,1})=2$.
\end{prop}
\begin{proof}
 Since $\mathcal B_{3}$ is the union of the tetrahedra $N L_{0} L_{1}
L_{2}$ and $S L_{0} L_{1} L_{2}$, the manifold $M_{3,1}$ has a
special spine with two true vertices and with only one
$2$-component. Hence $c(M_{3,1})=2$ (see~\cite{FrMartPetr03}).
\end{proof}

\begin{corollary}
 For every $n \geqslant 3$ we have $t(M_{n,k}) =
(-1)^{n} \varepsilon^{2-2n} + \varepsilon^{1-n} + 1$.
\end{corollary}

\end{document}